\documentclass[]{interact}

\usepackage{epstopdf}% To incorporate .eps illustrations using PDFLaTeX, etc.
\usepackage[caption=false]{subfig}% Support for small, `sub' figures and tables
\usepackage{anyfontsize}
\usepackage[numbers,sort&compress]{natbib}% Citation support using natbib.sty
\bibpunct[, ]{[}{]}{,}{n}{,}{,}% Citation support using natbib.sty
\renewcommand\bibfont{\fontsize{10}{12}\selectfont}% Bibliography support using natbib.sty

\theoremstyle{plain}% Theorem-like structures provided by amsthm.sty
\newtheorem{theorem}{Theorem}[section]
\newtheorem{lemma}[theorem]{Lemma}
\newtheorem{corollary}[theorem]{Corollary}

\theoremstyle{definition}
\newtheorem{assumption}{Assumption}

\theoremstyle{remark}

\usepackage{amsthm}
\usepackage{graphicx}
\usepackage[symbol]{footmisc}

\usepackage{float}
\usepackage{mathtools}
\usepackage{amssymb, amsmath, latexsym}
\usepackage{nicefrac}
\usepackage{hhline}
\usepackage{makecell}
\usepackage{multirow}
\usepackage[utf8]{inputenc} % allow utf-8 input
\usepackage[T1]{fontenc}    % use 8-bit T1 fonts
\usepackage{hyperref}
\usepackage{url}            % simple URL typesetting
\usepackage{booktabs}       % professional-quality tables
\usepackage{amsfonts}       % blackboard math symbols
\usepackage{nicefrac}       % compact symbols for 1/2, etc.
\usepackage{microtype}      % microtypography
\usepackage{xcolor}         % colors
\usepackage{amsmath,amsfonts,amssymb,amsthm,array}
\usepackage{bm}
\usepackage{color}
\usepackage{graphicx}
\usepackage{enumitem}
\usepackage{bbm}
\usepackage{relsize}
\usepackage{algorithm}
\usepackage{algpseudocode}
\usepackage{pifont}
\usepackage{tikz}
\allowdisplaybreaks

\usepackage{colortbl}
\definecolor{bgcolor}{rgb}{0.8,1,1}
\definecolor{bgcolor2}{rgb}{0.8,1,0.8}
\usepackage{threeparttable}

\usepackage{pifont}% http://ctan.org/pkg/pifont
\usepackage{algorithm}
\usepackage{algpseudocode}

\allowdisplaybreaks
\def \R {\mathbb R}

\def\R{\mathbb{R}}

\def\R{\mathbb R}

\def\<#1,#2>{\langle #1,#2\rangle}

\begin{document}
\title{Methods for Solving Variational Inequalities with Markovian Stochasticity}

\author{
\name{Vladimir Solodkin\textsuperscript{a}\thanks{CONTACT Vladimir Solodkin Email: solodkin.vs@phystech.su}, Michael Ermoshin\textsuperscript{a}, Roman Gavrilenko\textsuperscript{a}, Aleksandr Beznosikov\textsuperscript{a}}
\affil{\textsuperscript{a}Basic Research of Artificial Intelligence Laboratory (BRAIn Lab)}
}

\maketitle

\begin{abstract}
In this paper, we present a novel stochastic method for solving variational inequalities (VI) in the context of Markovian noise. By leveraging Extragradient technique, we can productively solve VI optimization problems characterized by Markovian dynamics. We demonstrate the efficacy of proposed method through rigorous theoretical analysis, proving convergence under quite mild assumptions of $L$-Lipschitzness, strong monotonicity of the operator and boundness of the noise only at the optimum. In order to gain further insight into the nature of Markov processes, we conduct the experiments to investigate the impact of the mixing time parameter on the convergence of the algorithm.
\end{abstract}


\section{Introduction}
Stochastic gradient methods are crucial for solving a wide range of optimization problems, with various applications in machine learning \cite{Goodfellow2014, Goodfellow2016}, including areas ranging from traditional empirical risk minimization \cite{vapnik} to modern reinforcement learning \cite{tomar2021mirrordescentpolicyoptimization, Schulman2015, lan2022policymirrordescentreinforcement}. The majority of minimization problems in machine learning have a stochastic structure, and are typically addressed through SGD-like approaches \cite{Cotter2011, Vaswani2019, Taylor2019, Aybat2019, gorbunov2019unified, Woodworth2021}. While such methods have been the focus of considerable research, most of the results surrounding the analysis of these methods rely heavily on the standard assumption of noise independence. However, this assumption become unrealistic in many practical scenarios. For instance, in distributed optimization, dependencies arise due to the communication delays and synchronization problems \cite{Lopes2007, Dimakis2010, Mao2020}. Similarly, in reinforcement learning \cite{Bhandari2018, Srikant2019, Durmus2021}, data is generated by interacting with an environment, leading to highly-correlated dependencies. These observations underscore the necessity for the development of novel methodologies that can operate effectively in the presence of non-i.i.d. stochasticity.

At present, the research on stochastic methods with Markovian noise in minimization problems is less developed in comparison to research on methods with independent noise. This inconsistency can be attributed to several factors. Notably, methods incorporating Markovian noise often present more intricate mathematical challenges \cite{Duchi2012, Beznosikov2023MarkovianNoise, Even2023}, as the dependencies between observations complicate the proof of fundamental algorithmic properties. In contrast, methods assuming independent noise generally have a more straightforward structure, facilitating their theoretical analysis (see \cite{gorbunov2019unified} and references therein). Nevertheless, the investigation of Markovian noise in these methods is a rapidly evolving field, with a growing body of literature addressing the challenges associated with minimizing objective functions under such conditions \cite{Wang2022, karimi2019nonasymptoticanalysisbiasedstochastic, Doan2023}.

The majority of previous works utilizing Markovian stochastics have focused on the minimization problem. However, from a practical point of view, more generalized frameworks like variational inequalities (VI) are also of significant interest. VIs provide a unified formulation for a broad class of problems, encapsulating a wide range of applications from game theory to traffic flow analysis \cite{Scutari2010, Jofre2007}. In particular, examples include economic market modeling, in which firms engage in competition and the objective is to identify a state of market balance; network routing, in which optimal paths are determined in the presence of competing flows; and auction theory, in which bidders strategies must converge to an equilibrium. Furthermore, solving saddle-point problems, which constitute a significant subset of variational inequality challenges, is a crucial aspect of training Generative Adversarial Networks (GANs) \cite{Gidel2018, Mertikopoulos2018, Chavdarova2019}. These problems often involve optimizing a min-max objective function, wherein the generator and discriminator assume adversarial roles. The theoretical foundation provided by VI frameworks is vital for ensuring convergence and stability in such adversarial settings. This approach can be effectively used not only in GANs training \cite{Liang2019}, but also in other applications where equilibrium conditions and competitive interactions prevail. As such, developing robust stochastic methods for VIs \cite{Juditsky2011, Mishchenko2020a}, especially in the presence of dependent noise, is an active area of research. These methods aim to extend the robustness and applicability of classical techniques to more complex and realistic settings, where noise dependencies are inherent and unavoidable.

In light of these insights, it becomes evident that there is a pressing need to develop a more practically applicable method. The objective of this paper is to incorporate the Markovian dynamics into the Extragradient algorithm \cite{Korpelevich1977}, thereby eliminating the limitations of traditional assumptions and expanding the scope of practical applications.
\subsection{Related Work}
Deterministic methods for solving VIs with a Lipschitz operator made significant progress with the advent of Extragradient method \cite{Korpelevich1977}. This classic method involves a two-step iterative process: first, an intermediate point is calculated using the current gradient operator, and then the actual update is performed by using operator at the intermediate point. This additional step is a meaningful improvement in the stability of convergence of the method, ultimately ensuring a more reliable optimization process. By introducing this, the method achieves greater robustness and improved performance in finding solutions. One of a more sophisticated modification of Extragradient method is Extrapolation From The Past \cite{Popov1980}. This single-call method calculates the operator only once per iteration, and the distinction lies in the definition of the intermediate point, which employs the operator in the previous intermediate point, not in the current, leading to halving the gradient calculations. Another significant approach is the Mirror-Prox algorithm \cite{Nemirovski2004}, which utilizes Bregman divergence to improve the optimization process. This idea adapts to the underlying geometry of the problem by employing Bregman divergence, permitting the use of non-Euclidean steps. By respecting the curvature of the solution space, it facilitates more efficient and stable convergence in complex, high-dimensional settings.

There were also significant developments in stochastic methods for solving variational inequalities. For instance, \cite{Juditsky2011} explored a stochastic version of the Mirror-Prox method, examining a scenario with constrained noise variance. This work served as the foundation for subsequent development in this field. To avoid the bounded variance assumption, \cite{Mishchenko2020a} proposed Revisiting Stochastic Extragradient, another stochastic modification of the Extragradient algorithm with the same randomness.
% In addition, a significant contribution was made by the Stochastic Variance Reduced Extragradient (SVRE) method \cite{Chavdarova2019}, which was further developed in subsequent studies with the introduction of VR-AGDA \cite{Yang2020}. In regard to the SVRE method \cite{Chavdarova2019}, it employs a combination of stochastic optimization, extragradient techniques, and variance reduction to enhance convergence rates in saddle-point problems. This method takes techniques of variance reduction, as exemplified by SVRG, which is used for solving minimization problems, and adapts them to the context of saddle-point problems by using the Extragradient methodology \cite{Korpelevich1977}. Besides these papers, it is worth mentioning work \cite{Alacaoglu2022}, which presents stochastic variations of both the Mirror-Prox and the Extragradient algorithms. The aforementioned methods \cite{Chavdarova2019, Yang2020, Alacaoglu2022} are unified by their reliance on variance reduction of the operator, thereby improving convergence rates. Stochastic methods, which minimize the number of gradient evaluations per iteration \cite{Kingma2014, Popov1980}, are particularly useful in large-scale settings where computational resources are a constraint. 
In case of less general settings, e.g. finite sum setup, the classical variance reduction technique is applicable. This is exemplified by the implementation of variance reduction in the context of variational inequalities, as demonstrated in the work \cite{Alacaoglu2022}, where variance reduction modifications of both classical Extragradient \cite{Korpelevich1977} and Mirror-Prox \cite{Juditsky2011} were presented. Prior to mentioned work \cite{Alacaoglu2022}, efforts were also made to adapt the same technique, as evidenced by \cite{Chavdarova2019, Yang2020}. Nevertheless, the majority of these results are predicated on the assumption of independent noise \cite{Palaniappan2016, Chavdarova2019, beznosikov2020distributed, Yang2020, Alacaoglu2022, Beznosikov2023, pichugin2023optimal, pichugin2024method}. 
% A notable exception is the work \cite{Beznosikov2023}, which considers Markovian noise in optimization problems, marking an important step towards more realistic noise assumptions.

Markovian noise, where the next iteration becomes dependent on the previous one, models real-world scenarios more accurately than i.i.d. noise scenario, capturing temporal dependencies and sequential decision processes. For instance, the work \cite{Beznosikov2023MarkovianNoise} provided a comprehensive framework for analyzing first-order gradient methods in stochastic minimization and VIs involving Markovian noise. Their approach achieved optimal linear dependence on the mixing time of the noise sequence in the stochastic term through a randomized batching scheme based on the multilevel Monte Carlo method. This technique eliminated several limiting assumptions from previous research, such as the need for a bounded domain and uniformly bounded stochastic operator. Notably, their extension to VIs under Markovian noise represented a significant contribution, providing matching lower bounds for oracle complexity in VI problems. In another paper \cite{Wang2022}, stochastic gradient-based Markov chain methods (MC-SGM) were analyzed for the min-max problem. The authors used algorithmic stability within the framework of statistical learning theory for both smooth and nonsmooth cases. However, the results in this paper are rather sparse and a continuation of this topic is needed.

\subsection{Our Contributions} 
The main contributions of this paper are the following:

$\bullet$ \textbf{Novel Stochastic Look at VI}\\
We present a novel view on variational inequalities through the lens of Markovian stochasticity. This approach differs from traditional methods such as stochastic Extragradient and its variations, which typically impose restrictions on noise independence \cite{Chavdarova2019, Palaniappan2016, Yang2020, Alacaoglu2022, Mishchenko2020a}. Among the works dealing with the Markovian noise, our paper is distinguishable by a mild assumption on the noise variables bounding the operator only at the optimum. However, we necessitate the Lipschitzness and strong monotonicity of the operator for all realizations of a random variable. On the other hand, analogous assumptions are present in the work \cite{Mishchenko2020a}, wherein the noise is considered to be independent.

$\bullet$ \textbf{Rates of Convergence}\\
We provide sharp rates of convergence, being able to avoid the presence of mixing time in the deterministic term of the rate.
The stochastic term is quadratic with respect to the mixing time of the underlying Markov chain, which is competitive with the exiting foundations in the literature. 
% Therefore, in comparison with work \cite{Beznosikov2023MarkovianNoise}, we have a superior deterministic term in terms of convergence, but a less optimal stochastic term.

$\bullet$ \textbf{Experimental Analysis of Mixing Time Influence}\\
In order to determine the actual convergence rate of the specified method and to verify the feasibility of the theoretical assessment in practice, we provide the numerical experiments. The aim of them is to demonstrate the effect of changes in the mixing time parameter on the convergence process.
% \subsection{Our Contributions}

% Our main contributions are the following:
% \begin{enumerate}
%     \item \textbf{Novel Stochastic Methods}: We introduce stochastic extragradient methods designed for solving VIs in the context of Markov dynamics.
%     \item \textbf{Theoretical Analysis}: We provide rigorous proofs of convergence under standard assumptions of smoothness, strong convexity, and strong monotonicity.
%     % \item \textbf{Algorithm Development}: We develop and analyze the Extrapolated Gradient Method, proving its effectiveness for VIs with Markovian properties.
%     % \item \textbf{Application to GANs}: We demonstrate the applicability of our methods to GAN training, showcasing improved stability and performance.
%     \item \textbf{Experimental Validation}: We validate our theoretical findings with extensive experiments, highlighting the practical benefits of our approach.
% \end{enumerate}

\subsection{Notation}
We use $\langle x, y\rangle := \sum_{i = 1}^d x_i y_i$ to denote standard inner product of vectors $x, y \in \mathbb{R}^d$. We introduce $l_2$-norm of vector $x \in \mathbb{R}^d$ as $\|x\| := \sqrt{\langle x, x\rangle}$. Let $(\mathsf{M},\mathsf{d}_{\mathsf{M}})$ be a complete separable metric space endowed with its Borel $\sigma$-field $\mathcal{M}$.
{We denote by $(\mathsf{M}^{\ensuremath{\mathbb{N}}}, \mathcal{M}^{\otimes \ensuremath{\mathbb{N}}})$ the corresponding canonical process. Consider the Markov kernel defined on $\mathsf{M} \times \mathcal{M}$, and denote by $\mathbb{P}_{\theta}$ and $\mathbb{E}_{\theta}$ the corresponding probability distribution and the expected value with initial distribution $\theta$. Let $(M_k)_{k \in \ensuremath{\mathbb{N}}}$ be the corresponding canonical process. For $\theta = \delta_{m}$, $m \in \mathsf{M}$, we simply write $\mathbb{P}_{m}$ and $\mathbb{E}_{m}$ instead of $\mathbb{P}_{\delta_{m}}$ and $\mathbb{E}_{\delta_{m}}$.}
For $z^{0},\ldots,z^{t}$ being the iterates of any algorithm, we denote $\mathcal{F}_{t} = \sigma(z^{j}, j \leq t)$ and write $\mathbb{E}_{t}$ to denote conditional expectation $\mathbb{E}[\cdot | \mathcal{F}_{t}]$.

% By maintaining consistent notation, we guarantee clarity and facilitate understanding of the algorithms and theoretical results presented in this paper. 
% s takim malen'kim notation eta fraza kak budto lishnyaya

\section{Technical preliminaries}
In this paper, we are interested in solving the optimization problem of the form
\begin{equation}
    \label{eq:main}
    \text{Find } \, z^* \in \mathbb{R}^d \, \text{ such that } \forall z \in \mathbb{R}^d \hookrightarrow \langle F(z^*), z - z^* \rangle \geq 0,
\end{equation}

where $F(z) := \mathbb{E}_{\xi \sim \pi} F(z, \xi)$ is approximated by the stochastic oracle $F(z, \xi)$ and $\{ \xi^t \}_{t=0}^{\infty}$ is a stationary Markov chain with a unique invariant distribution $\pi$, defined on $\mathcal{M}$. {Since the feasible set is Euclidean space $\mathbb{R}^d$, the problem \eqref{eq:main} is equivalent to finding $z^*$ such that $F(z^*) = 0$.} The aforementioned formulation of the VI problem is a classical approach in optimization methods.

We begin by introducing two fundamental constraints on the operator $F(\cdot, \xi)$.

\begin{assumption}[$L(\xi)$-Lipschitzness]
\label{as:lip}
    The operator $F(\cdot, \xi)$ is $L(\xi)$-Lipschitz, i.e., there exists $L(\xi) > 0$ such that the following inequality holds for all $z', z'' \in \mathbb{R}^d$:
    \begin{equation*}
       \| F (z', \xi) - F (z'', \xi) \| \leq L(\xi)  \| z' - z''\|.
    \end{equation*}
    We also define $L := \sup\limits_{\xi \in \mathcal{M}} L(\xi) < +\infty.$
\end{assumption}

\begin{assumption}[$\mu(\xi)$-strong monotonicity]
    \label{as:stronglymonotone}
        The operator $F(\cdot, \xi)$ is $\mu(\xi)$-strongly monotone, i.e., there exists $\mu(\xi) > 0$ such that the following inequality holds for all $z', z'' \in \mathbb{R}^d$:
        \begin{equation*}
            \langle F(z', \xi)-F(z'', \xi) , z'-z'' \rangle\geq \mu(\xi) \|z'-z''\|^2.
        \end{equation*} 
        We also define $\mu := \inf\limits_{\xi \in \mathcal{M}} \mu(\xi) > 0.$
\end{assumption}

{Assumptions \ref{as:lip} and \ref{as:stronglymonotone} imply that the averaged operator $F$ is also $L$-Lipschitz and $\mu$-strongly monotone. Indeed, Lipschitzness holds because norm of an expectation is bounded by expectation of the norm, and strong monotonicity survives because the inner product is linear. Hence, the problem \eqref{eq:main} admits a unique solution $z^*$ \cite{ref1}.}

The next assumption implies a uniform stochastic boundedness of gradient operator at the optimum point.

\begin{assumption}
    \label{as:optimumlimited}
        The oracle $F(z, \cdot)$ is bounded at the optimum $z^*$, i.e., there exists $\sigma_* > 0$ such that the following inequality holds for all $\xi \in \mathcal{M}$:
        \begin{equation*}
            \| F(z^*, \xi) \| \leq \sigma_*.
        \end{equation*}
\end{assumption}

To the best of our knowledge, this constraint seems to be unpopular in the majority of existing literature. Most of the existing work assumes either uniform boundness or boundness of the moments of distribution.

Now we introduce an important assumption related to the theory of Markov processes.
\begin{assumption}[{uniform geometric ergodicity}]
    \label{as:var}
     {Let $\{\xi^t\}_{t=0}^{\infty}$ be a stationary Markov chain on $(\mathsf{M},\mathcal{M})$ with Markov kernel $\mathcal{Q}$ and a unique invariant distribution $\pi$. Assume that $\{\xi^t\}_{t=0}^{\infty}$ is uniformly geometrically ergodic with mixing time $\tau_{\mathrm{mix}}$, i.e. for all $t>0$}
    \begin{equation*}
        {\sup_{m,m' \in \mathsf{M}}\| \mathcal{Q}^t(m, \cdot) - \mathcal{Q}^t(m', \cdot) \|_{\mathrm{TV}} \leq 2^{-t / \tau_{\mathrm{mix}}}.}
    \end{equation*}
\end{assumption}
{By convexity of the total variation norm, stationarity of $\pi$ and Jensen's inequality, this assumption implies}
\begin{equation*}
    {\sup_{m \in \mathsf{M}}\| \mathcal{Q}^t(m, \cdot) - \pi \|_{\mathrm{TV}} \leq 2^{-t / \tau_{\mathrm{mix}}}.}
\end{equation*}
This kind of assumption is standard concerning the literature on Markovian noise \cite{Chavdarova2019, Palaniappan2016, Yang2020, Alacaoglu2022}. It is quite obvious from definition that the mixing time $\tau_{mix}$ is simply the number of steps of the Markov chain required for the distribution of the current state to be sufficiently close to the stationary probability $\pi$.

\section{Main results}

Now, we are ready to introduce our Algorthm \ref{alg:EGM}. Following the idea of utilizing the intermediate step information, we incorporate the Markovian stochasticity into the classical \texttt{Extrapolated Gradient Method} \cite{Korpelevich1977}. An extrapolation step aims at stabilizing the convergence process, while the Markovian approach pursues to utilize the natural idea of using the previous time iterations. The aforestated algorithm outlines the steps involved.

\begin{algorithm}[h]
   \caption{\texttt{Extrapolated Gradient Method with Markovian noise}}
   \label{alg:EGM}
        \begin{algorithmic}[1]
           \State {\bfseries Parameters:} step size $\gamma > 0$, number of iterations $T$
           \State {\bfseries Initialization:} choose $z^0 \in \mathcal{Z}$
           \For{$t = 0$ {\bfseries to} $T$}
           \State $z^{t+\frac{1}{2}} = z^t - \gamma F(z^t, \xi^t)$
           \State $z^{t+1} = z^t - \gamma F(z^{t+\frac{1}{2}}, \xi^{t})$
           \EndFor
        \end{algorithmic}
\end{algorithm}
In order to prove the main convergence theorem of Algorithm \ref{alg:EGM}, we first establish three important lemmas. One of them is responsible for handling the deterministic step of the proof, while the other two aim to circumvent the difficulties introduced by the Markovian nature of stochasticity. Let us start with the classical \cite{Gidel2018, Mishchenko2020a, Hsieh2019} descent lemma:

\begin{lemma}
\label{lem:1}
    Let Assumptions \ref{as:lip}, \ref{as:stronglymonotone} be satisfied. Then for the iterates $\{z_t\}_{t\geq0}$ of Algorithm \ref{alg:EGM} it holds that:
    \begin{align*}
        \| z^{t+1} - z^*\|^2 &\leq \|z^t - z^*\|^2 - 2 \gamma \mu \| z^{t+\frac{1}{2}} - z^* \|^2 - 2 \gamma \langle F(z^*, \xi^t), z^{t + \frac{1}{2}} - z^* \rangle \\&\quad+ \gamma^2 L^2 \| z^t - z^{t+\frac{1}{2}} \|^2 - \| z^t - z^{t+\frac{1}{2}}\|^2.
    \end{align*}
\end{lemma}

\begin{proof}
    We start by using line 5 of Algorithm \ref{alg:EGM}:
    \begin{align*}
        \| z^{t + 1} - z^*\|^2 & = \| z^t - \gamma F(z^{t+\frac{1}{2}}, \xi^t) - z^* \|^2 \\
        & = \| z^t - z^*\|^2 - 2 \gamma \langle z^t - z^*, F(z^{t+\frac{1}{2}}, \xi^t) \rangle + \| \gamma F(z^{t+\frac{1}{2}}, \xi^t)\|^2 \\
        & = \| z^t - z^*\|^2 - 2 \gamma \langle z^{t+\frac{1}{2}} - z^*, F(z^{t+\frac{1}{2}}, \xi^t) \rangle \\
        & \quad+ 2 \gamma \langle z^{t+\frac{1}{2}} - z^t, F(z^{t+\frac{1}{2}}, \xi^t) \rangle +
        \| \gamma F(z^{t+\frac{1}{2}}, \xi^t)\|^2.
    \end{align*}
    Note that
    \begin{align*}
       \| \gamma F(z^{t+\frac{1}{2}}, \xi^t)\|^2 &= \| \gamma F(z^{t+\frac{1}{2}}, \xi^t) - \gamma F(z^t, \xi^t) \|^2 - \| \gamma F(z^t, \xi^t) \|^2 \\&\quad+ 2 \langle \gamma F(z^{t+\frac{1}{2}}, \xi^t), \gamma F(z^t, \xi^t) \rangle.
    \end{align*}
    Using this and the fact $z^{t+\frac{1}{2}} - z^t = - \gamma F(z^t, \xi^t)$, we have
    \begin{align*}
        \| z^{t + 1} - z^*\|^2  &= \| z^t - z^*\|^2 - 2 \gamma \langle F(z^{t+\frac{1}{2}}, \xi^t), z^{t+\frac{1}{2}} - z^* \rangle \\
        &\quad- 2 \langle \gamma F(z^t, \xi^t), \gamma F(z^{t+\frac{1}{2}}, \xi^t) \rangle \\
        & \quad+ \| \gamma F(z^{t+\frac{1}{2}}, \xi^t) - \gamma F(z^t, \xi^t) \|^2 - \| \gamma F(z^t, \xi^t) \|^2\\&\quad + 2 \langle \gamma F(z^{t+\frac{1}{2}}, \xi^t), \gamma F(z^t, \xi^t) \rangle \\
        & = \| z^t - z^*\|^2 - 2 \gamma \langle F(z^{t+\frac{1}{2}}, \xi^t), z^{t+\frac{1}{2}} - z^* \rangle - \| z^{t+\frac{1}{2}} - z^t \|^2 \\
        & \quad+ \gamma^2 \| F(z^{t+\frac{1}{2}}, \xi^t) - F(z^t, \xi^t) \|^2.
    \end{align*}
    Next, we use Assumption \ref{as:lip}:
    \begin{align}
        \label{lem1:tech}
        \| z^{t + 1} - z^*\|^2 &\leq \| z^t - z^*\|^2 - 2 \gamma \langle F(z^{t+\frac{1}{2}}, \xi^t), z^{t+\frac{1}{2}} - z^* \rangle \\
        &
        \quad+ \gamma^2 L^2 \| z^{t+\frac{1}{2}} - z^t \|^2 - \| z^{t+\frac{1}{2}} - z^t \|^2.\notag
    \end{align}
    Applying Assumption \ref{as:stronglymonotone} to the second term, we get 
    \begin{align*}
        \langle F(z^{t+\frac{1}{2}}, \xi^t), z^{t+\frac{1}{2}} - z^* \rangle & = \langle F(z^{t+\frac{1}{2}}, \xi^t) - F(z^*, \xi^t), z^{t+\frac{1}{2}} - z^* \rangle \\
        &\quad+ \langle F(z^*, \xi^t), z^{t+\frac{1}{2}} - z^* \rangle \\
        & \geq \mu \| z^{t+\frac{1}{2}} - z^* \|^2 + \langle F(z^*, \xi^t), z^{t+\frac{1}{2}} - z^* \rangle
    \end{align*}
    Substituting this into \eqref{lem1:tech} completes the proof.
\end{proof}
One of the effective methods for addressing Markovian stochasticity is the technique of stepping back by $\mathcal{T} > \tau_{mix}$ steps. In contrast to the i.i.d. case, in which the term of the form $\mathbb{E}\langle F(z^*, \xi^t), z^{t + \frac{1}{2}} - z^*\rangle$ equals to zero, we here need to handle it in the following way:           
\begin{align}
    \label{lemms23}
    \langle F(z^*, \xi^t), z^{t+\frac{1}{2}} - z^* \rangle & = 
    \langle F(z^*, \xi^t), z^{t+\frac{1}{2}} - z^{t+\frac{1}{2}-\mathcal{T}}\rangle + \langle F(z^*, \xi^t), z^{t+\frac{1}{2}-\mathcal{T}} - z^* \rangle,
\end{align}
    taking what is referred to as the step back. The first term in \eqref{lemms23} is treated with the help of Cauchy-Schwarz inequality and the lemma below.  
\begin{lemma}
\label{lem:2}
    Let $\mathcal{T} \in \mathbb{N}$ be a fixed number and let $\{z_t\}_{t\geq0}$ be the iterates of Algorithm \ref{alg:EGM}. Then, for all $t \geq \mathcal{T}$ it holds that
    \begin{align*}
        \| z^{t+\frac{1}{2}} - z^{t+\frac{1}{2}-\mathcal{T}} \| \leq \sum\limits_{k=0}^{\mathcal{T}} [1 + \gamma L] \| z^{t+\frac{1}{2}-k} - z^{t-k} \|.
    \end{align*}
\end{lemma}
\begin{proof}
    We start with the triangle inequality:
    \begin{align*}
        \| z^{t+\frac{1}{2}}  - z^{t+\frac{1}{2}-\mathcal{T}} \| 
        & \leq \| z^{t+\frac{1}{2}} - z^t \| + \| [ z^t - z^{t-1} ] - [ z^{t-\frac{1}{2}} - z^{t-1} ] \| \\
        &\quad+ \| z^{t-\frac{1}{2}} - z^{t+\frac{1}{2}-\mathcal{T}} \|.
    \end{align*}
    Using lines 4 and 5 of Algorithm \ref{alg:EGM}, we obtain:
    \begin{align*}
        \| z^{t+\frac{1}{2}}  - z^{t+\frac{1}{2}-\mathcal{T}} \| 
        & \leq \| z^{t+\frac{1}{2}} - z^t \| + \gamma \| F(z^{t-\frac{1}{2}}, \xi^{t-1}) - F(z^{t-1}, \xi^{t-1}) \| \\
        &\quad+ \| z^{t-\frac{1}{2}} - z^{t+\frac{1}{2}-\mathcal{T}} \| \\
        & \leq \| z^{t+\frac{1}{2}} - z^t \| + \gamma L \| z^{t-\frac{1}{2}} - z^{t-1} \| + \| z^{t-\frac{1}{2}} - z^{t+\frac{1}{2}-\mathcal{T}} \|,
    \end{align*}
    where in the last inequality follows from Assumption \ref{as:lip}. The same steps can now be applied to the $\| z^{t-\frac{1}{2}} - z^{t+\frac{1}{2}-\mathcal{T}} \|$ term. Finally, after performing the recursion, we obtain 
    \begin{align*}
        \| z^{t+\frac{1}{2}} - z^{t+\frac{1}{2}-\mathcal{T}} \| &\leq \| z^{t+\frac{1}{2}} - z^t \| + \sum\limits_{k=1}^{\mathcal{T} - 1} [1 + \gamma L] \| z^{t+\frac{1}{2}-k} - z^{t-k} \| \\
        &\quad+ \gamma L \| z^{t+\frac{1}{2} - \mathcal{T}} - z^{t-\mathcal{T}} \|.
    \end{align*}
    This concludes the proof.
\end{proof}
The second term in \eqref{lemms23} is resolved with the following lemma. 
\begin{lemma}
\label{lem:3}
    {For any $\mathcal{T} > \tau_{mix}$, any $t > \mathcal{T}$, and any $z^{t+\frac{1}{2}-\mathcal{T}} \in \mathbb{R}^d$ measurable with respect to $\mathcal{F}_{t+\frac{1}{2}-\mathcal{T}}$, it holds that}
\[
{\mathbb{E}\left[ \left\langle F(z^*, \xi^t)-F(z^*), z^*-z^{t + \frac{1}{2} - \mathcal{T}}  \right\rangle \right] \leq 2\sigma_* \cdot 2^{-\mathcal{T}/\tau_{\mathrm{mix}}} \, \mathbb{E} \left[\|z^*-z^{t + \frac{1}{2} - \mathcal{T}}\| \right].}
\]
\end{lemma}
\begin{proof}
    {Let $v := z^*-z^{t + \frac{1}{2} - \mathcal{T}}$, which is $\mathcal{F}_{t+\frac{1}{2}-\mathcal{T}}$-measurable. By the tower property,}
    \begin{equation*}
    {\mathbb{E}\left[ \left\langle F(z^*, \xi^t)-F(z^*), v \right\rangle \right] = \mathbb{E}\left[ \left\langle \mathbb{E}_{t+\frac{1}{2}-\mathcal{T}}\left[F(z^*, \xi^t)-F(z^*)\right], v \right\rangle \right].}
    \end{equation*}
    {Conditioned on $\mathcal{F}_{t+\frac{1}{2}-\mathcal{T}}$, the law of $\xi^t$ is $\mathcal{Q}^{\mathcal{T}}(\xi^{t-\mathcal{T}},\cdot)$. Therefore,}
    \begin{equation*}
    {\mathbb{E}_{t+\frac{1}{2}-\mathcal{T}}\left[F(z^*, \xi^t)-F(z^*)\right] = \int_{\mathsf M} F(z^*,m)\,\big(\mathcal{Q}^{\mathcal{T}}(\xi^{t-\mathcal{T}},dm)-\pi(dm)\big).}
    \end{equation*}
    {Using the dual characterization of total variation distance, for any bounded measurable function $g$ one has}
    \begin{equation*}
    {\left\|\int g(m)\,(\nu-\nu')(dm)\right\| \leq 2\|g\|_{\infty}\,\|\nu-\nu'\|_{\mathrm{TV}}.}
    \end{equation*}
    {Applying this with $g(m)=F(z^*,m)$ and using Assumptions \ref{as:optimumlimited} and \ref{as:var}, we obtain}
    \begin{equation*}
    {\left\|\mathbb{E}_{t+\frac{1}{2}-\mathcal{T}}\left[F(z^*, \xi^t)-F(z^*)\right]\right\| \leq 2\sigma_*\,\|\mathcal{Q}^{\mathcal{T}}(\xi^{t-\mathcal{T}},\cdot)-\pi\|_{\mathrm{TV}} \leq 2\sigma_* \cdot 2^{-\mathcal{T}/\tau_{\mathrm{mix}}}.}
    \end{equation*}
    {Finally, by Cauchy--Schwarz inequality,}
    \begin{equation*}
    {\mathbb{E}\left[ \left\langle F(z^*, \xi^t)-F(z^*), z^*-z^{t + \frac{1}{2} - \mathcal{T}} \right\rangle \right] \leq 2\sigma_* \cdot 2^{-\mathcal{T}/\tau_{\mathrm{mix}}}\,\mathbb{E}\|z^*-z^{t + \frac{1}{2} - \mathcal{T}}\|,}
    \end{equation*}
    {which proves the claim.}
\end{proof}
Now we are ready to prove the convergence of Algorithm \ref{alg:EGM}. The proof scheme is to utilize Lemmas \ref{lem:1} and \ref{lem:2} and Cauchy-Schwarz inequalities, proceed by taking the full mathematical expectation on both sides of the resulting expressions and, with the aid of Lemma \ref{lem:3}, sum these expectations from some index $\tau > \tau_{mix}$ up to iteration $T - 1$. Through these steps and by imposing specific conditions on the parameters $\gamma$, we derive the final convergence result.
\begin{theorem}[Convergence of Algorithm \ref{alg:EGM}]
\label{theorem:1}
Let Assumptions \ref{as:lip}, \ref{as:stronglymonotone}, \ref{as:optimumlimited}, \ref{as:var} be satisfied. Let the problem \eqref{eq:main} be solved by Algorithm \ref{alg:EGM}. {Than, for all $T \gtrsim \tau_{mix}$ and $0 \leq \gamma \leq \min\{\frac{1}{2L}, \frac{1}{4\mu}\}$, it holds that}
\begin{equation*}
    \mathbb{E} \|z^{T+1} - z^*\|^2 \leq \Bigg(1 - \frac{\mu\gamma}{2}\Bigg)^T\Bigg[\Big(1-\frac{\mu\gamma}{2}\Big)^{-{\tau_{mix}}}\mathbb{E}\|z^{\tau_{mix}} - z^*\|^2 + \Delta_{\tau_{mix}}\Bigg] + \frac{50\gamma{\tau_{mix}}^2}{\mu}\sigma_*^2,
\end{equation*}
where $\Delta_{\tau_{mix}} := \sum\limits_{t=0}^{{\tau_{mix}}-1}\mathbb{E}\Big[ \|z^{t+\frac{1}{2}} - z^t\|^2 + \|z^{t+\frac{1}{2}} - z^*\|^2\Big]$.
\end{theorem}

\begin{proof}
    We start by using Lemma \ref{lem:1}:
    \begin{align*}
        \| z^{t+1} - z^*\|^2 &\leq \|z^t - z^*\|^2 - 2 \gamma \mu \| z^{t+\frac{1}{2}} - z^* \|^2 - 2 \gamma \langle F(z^*, \xi^t), z^{t + \frac{1}{2}} - z^* \rangle \\&\quad+ \gamma^2 L^2 \| z^t - z^{t+\frac{1}{2}} \|^2 - \| z^t - z^{t+\frac{1}{2}}\|^2.
    \end{align*}
    % Next, using strong monotonicity of $F(\cdot, \xi^t)$ (Assumption \ref{as:stronglymonotone}) we can bound $\langle F(z^{t+\frac{1}{2}}, \xi^t), z^{t+\frac{1}{2}} - z^* \rangle$:
    Take a look at the second term. We use the stepping back technique:
    \begin{align*}
        \langle F(z^*, \xi^t), z^{t+\frac{1}{2}} - z^* \rangle & =   \langle F(z^*, \xi^t), z^{t+\frac{1}{2}-\tau} - z^* \rangle + \langle F(z^*, \xi^t), z^{t+\frac{1}{2}} - z^{t+\frac{1}{2}-\tau}\rangle, 
    \end{align*}
    where $\tau$ is an arbitrary number such that $\tau \geq \tau_{mix}$. Using the last bound, one can obtain:
    \begin{align*}
        \| z^{t + 1} - z^*\|^2  &\leq \| z^t - z^*\|^2 - 2 \gamma \mu \| z^{t+\frac{1}{2}} - z^* \|^2 - 2 \gamma \langle F(z^*, \xi^t), z^{t+\frac{1}{2}-\tau} - z^* \rangle \\
        &\quad - 2 \gamma \langle F(z^*, \xi^t), z^{t+\frac{1}{2}} - z^{t+\frac{1}{2}-\tau} \rangle - \| z^{t+\frac{1}{2}} - z^t \|^2 \\
        &\quad+ \gamma^2 L^2 \| z^{t+\frac{1}{2}} - z^t \|^2\\
        & = \| z^t - z^*\|^2 - 2 \gamma \mu \| z^{t+\frac{1}{2}} - z^* \|^2 + 2 \gamma \langle F(z^*, \xi^t), - z^{t+\frac{1}{2}-\tau} + z^* \rangle \\
        &\quad + 2 \gamma \langle F(z^*, \xi^t), - z^{t+\frac{1}{2}} + z^{t+\frac{1}{2}-\tau}\rangle - \| z^{t+\frac{1}{2}} - z^t \|^2\\
        &\quad+ \gamma^2 L^2 \| z^{t+\frac{1}{2}} - z^t \|^2.
    \end{align*}
    Now we use Cauchy-Schwarz inequality \eqref{eq:cbs2}:
    \begin{align*}
        \| z^{t + 1} &- z^*\|^2  \leq \| z^t - z^*\|^2 - 2 \gamma \mu \| z^{t+\frac{1}{2}} - z^* \|^2 - 2 \gamma \langle F(z^*, \xi^t), z^{t+\frac{1}{2}-\tau} - z^* \rangle \\
        & + 2 \gamma \| F(z^*, \xi^t) \| \, \| z^{t+\frac{1}{2}} - z^{t+\frac{1}{2}-\tau} \| + \gamma^2 L^2 \| z^{t+\frac{1}{2}} - z^t \|^2 - \| z^{t+\frac{1}{2}} - z^t \|^2.
    \end{align*}
    Applying Lemma \ref{lem:2} and Assumption \ref{as:optimumlimited}, we get:
    \begin{align*}
        \| z^{t + 1} &- z^*\|^2 \leq \| z^t - z^*\|^2 - 2 \gamma \mu \| z^{t+\frac{1}{2}} - z^* \|^2 - 2 \gamma \langle F(z^*, \xi^t), z^{t+\frac{1}{2}-\tau} - z^* \rangle \\
        & + 2 \gamma \sigma_* \sum\limits_{k=0}^{\tau} [1 + \gamma L] \| z^{t+\frac{1}{2}-k} - z^{t-k} \| + \gamma^2 L^2 \| z^{t+\frac{1}{2}} - z^t \|^2 - \| z^{t+\frac{1}{2}} - z^t \|^2.
    \end{align*}
    Using Cauchy-Schwarz inequality \eqref{eq:cbs1} with $\beta  = \beta_1$ to be specified later, one can obtain:
    \begin{align*}
        \| z^{t + 1}  - z^*\|^2 &\leq \| z^t - z^*\|^2 - 2 \gamma \mu \| z^{t+\frac{1}{2}} - z^* \|^2 - 2 \gamma \langle F(z^*, \xi^t), z^{t+\frac{1}{2}-\tau} - z^* \rangle \\
        &\quad+ \frac{\gamma}{\beta_1} {\sigma_*}^2 + \gamma(1+\gamma L)^2\beta_1 \bigg[ \sum\limits_{k=0}^{\tau} \| z^{t+\frac{1}{2}-k} - z^{t-k} \| \bigg]^2 \\
        &\quad+ \gamma^2 L^2 \| z^{t+\frac{1}{2}} - z^t \|^2 - \| z^{t+\frac{1}{2}} - z^t \|^2.
    \end{align*}
    We now take the full mathematical expectation from both sides of the last inequality:
    \begin{align*}
        \mathbb{E}\| z^{t + 1}  - z^*\|^2 &\leq \mathbb{E}\| z^t - z^*\|^2 - 2 \gamma \mu \mathbb{E}\| z^{t+\frac{1}{2}} - z^* \|^2 - \mathbb{E}\| z^{t+\frac{1}{2}} - z^t \|^2 \\
        &\quad+ \frac{\gamma}{\beta_1} {\sigma_*}^2 + \gamma(1+\gamma L)^2\beta_1 \mathbb{E}\bigg[ \sum\limits_{k=0}^{\tau} \| z^{t+\frac{1}{2}-k} - z^{t-k} \| \bigg]^2 \\
        &\quad+ \gamma^2 L^2 \mathbb{E}\| z^{t+\frac{1}{2}} - z^t \|^2 - 2 \gamma \mathbb{E} \langle F(z^*, \xi^t), z^{t+\frac{1}{2}-\tau} - z^* \rangle.
    \end{align*}
    Using the convexity of the squared norm, we get:
    \begin{align*}
        \mathbb{E}\| z^{t + 1}  - z^*\|^2 &\leq \mathbb{E}\| z^t - z^*\|^2 - 2 \gamma \mu \mathbb{E}\| z^{t+\frac{1}{2}} - z^* \|^2 - \mathbb{E}\| z^{t+\frac{1}{2}} - z^t \|^2 \\
        &\quad+ \frac{\gamma}{\beta_1} {\sigma_*}^2 + \gamma(1+\gamma L)^2\beta_1 \tau \sum\limits_{k=0}^{\tau} \mathbb{E}\| z^{t+\frac{1}{2}-k} - z^{t-k} \|^2 \\
        &\quad + \gamma^2 L^2 \mathbb{E}\| z^{t+\frac{1}{2}} - z^t \|^2 - 2 \gamma \mathbb{E} \langle F(z^*, \xi^t), z^{t+\frac{1}{2}-\tau} - z^* \rangle.
    \end{align*}
    Applying Lemma \ref{lem:3} to the last inequality leads to:
    \begin{align*}
        \mathbb{E}\| z^{t + 1}  - z^*\|^2 &\leq \mathbb{E}\| z^t - z^*\|^2 - 2 \gamma \mu \mathbb{E}\| z^{t+\frac{1}{2}} - z^* \|^2 + {4 \gamma 2^{-\tau/\tau_{\mathrm{mix}}} \sigma_* \mathbb{E} \| z^{t+\frac{1}{2}-\tau} - z^* \|} \\
        &\quad+ \frac{\gamma}{\beta_1} {\sigma_*}^2 + \gamma(1+\gamma L)^2\beta_1 \tau \sum\limits_{k=0}^{\tau} \mathbb{E}\| z^{t+\frac{1}{2}-k} - z^{t-k} \|^2 \\
        &\quad+ \gamma^2 L^2 \mathbb{E}\| z^{t+\frac{1}{2}} - z^t \|^2 - \mathbb{E}\| z^{t+\frac{1}{2}} - z^t \|^2.
    \end{align*}
    Using Cauchy-Schwarz inequality \eqref{eq:cbs1} with {$\beta = \gamma$}, we obtain:
    \begin{align*}
        \mathbb{E}\| z^{t + 1}  - z^*\|^2 &\leq \mathbb{E}\| z^t - z^*\|^2 - 2 \gamma \mu \mathbb{E}\| z^{t+\frac{1}{2}} - z^* \|^2 + {2^{-\tau/\tau_{\mathrm{mix}}+1} \mathbb{E} \| z^{t+\frac{1}{2}-\tau} - z^* \|^2} \\
        &\quad + {\left(\frac{\gamma}{\beta_1} + 2\gamma^2 2^{-\tau/\tau_{\mathrm{mix}}}\right) {\sigma_*}^2} + \gamma(1+\gamma L)^2\beta_1 \tau \sum\limits_{k=0}^{\tau} \mathbb{E}\| z^{t+\frac{1}{2}-k} - z^{t-k} \|^2 \\
        &\quad+ (\gamma^2 L^2 - 1) \mathbb{E}\| z^{t+\frac{1}{2}} - z^t \|^2.
    \end{align*}
    For $t \geq 0$, let $p_t = p^t$ and $p = (1 - \mu \gamma / 2)^{-1}$. Here we multiply the above expression by $p_t$ and sum for $\tau \leq t < T$, hoping for cancellations:
    \begin{align*}
        \sum\limits_{t=\tau}^{T-1} p_t\mathbb{E}\| z^{t + 1}  - z^*\|^2 &\leq \sum\limits_{t=\tau}^{T-1}p_t\mathbb{E}\| z^t - z^*\|^2 - 2 \gamma \mu \sum\limits_{t=\tau}^{T-1} p_t\mathbb{E}\| z^{t+\frac{1}{2}} - z^* \|^2 \\
        &\quad+ {\left(\frac{\gamma}{\beta_1} + 2\gamma^2 2^{-\tau/\tau_{\mathrm{mix}}}\right) {\sigma_*}^2}\sum\limits_{t=\tau}^{T-1} p_t \\
        &\quad + {2^{-\tau/\tau_{\mathrm{mix}}+1}}\sum\limits_{t=\tau}^{T-1} p_t\mathbb{E} \| z^{t+\frac{1}{2}-\tau} - z^* \|^2\\
        &\quad + (\gamma^2 L^2 - 1) \sum\limits_{t=\tau}^{T-1} p_t\mathbb{E}\| z^{t+\frac{1}{2}} - z^t \|^2 \\
        &\quad+ \gamma(1+\gamma L)^2\beta_1 \tau \sum\limits_{t=\tau}^{T-1} p_t\sum\limits_{k=0}^{\tau} \mathbb{E}\| z^{t+\frac{1}{2}-k} - z^{t-k} \|^2.
    \end{align*}
    Using the fact that 
    $(1 - a/x)^{-x} \leq 2 e^{a} \leq 2e$ for any $x \geq 2$ and $0 \leq a \leq 1$, we can estimate  
    $p_\tau = (1 - \mu \gamma_1 /(2 \tau))^{-\tau} \leq 2e \leq 6$, where $\gamma_1 = \gamma\tau, \gamma \leq (\mu\tau)^{-1}$. Then, we can rearrange the sums:
    \begin{align*}\sum\limits_{t = \tau}^{T-1} p_{t} \sum\limits_{k=0}^{\tau} \mathbb{E}\| z^{t+\frac{1}{2}-k} - z^{t-k} \|^2 &\leq 
    p^{\tau}\sum\limits_{t = \tau}^{T-1} \sum\limits_{k=0}^{\tau} {p_{t-k}} \mathbb{E}\| z^{t+\frac{1}{2}-k} - z^{t-k} \|^2 \\
    &\leq 
    6 \tau \sum\limits_{t = 0}^{T-1} p_{t} \mathbb{E}\| z^{t+\frac{1}{2}} - z^{t} \|^2 .\end{align*}
    Now we can estimate:
    \begin{align}
        \label{eq:th1}
        \sum\limits_{t=\tau}^{T-1} p_t\mathbb{E}\| z^{t + 1}  - z^*\|^2 &\leq \sum\limits_{t=\tau}^{T-1}p_t\mathbb{E}\| z^t - z^*\|^2 - 2 \gamma \mu \sum\limits_{t=\tau}^{T-1} p_t\mathbb{E}\| z^{t+\frac{1}{2}} - z^* \|^2 \notag\\
        &\quad+ {\left(\frac{\gamma}{\beta_1} + 2\gamma^2 2^{-\tau/\tau_{\mathrm{mix}}}\right) {\sigma_*}^2}\sum\limits_{t=\tau}^{T-1} p_t \notag\\
        &\quad+ {2^{-\tau/\tau_{\mathrm{mix}}+1}}\sum\limits_{t=\tau}^{T-1} p_t\mathbb{E} \| z^{t+\frac{1}{2}-\tau} - z^* \|^2\\
        &\quad + (\gamma^2 L^2 - 1) \sum\limits_{t=\tau}^{T-1} p_t\mathbb{E}\| z^{t+\frac{1}{2}} - z^t \|^2 \notag\\
        &\quad+ 6\gamma(1+\gamma L)^2\beta_1 \tau^2 \sum\limits_{t = 0}^{T-1} p_{t} \mathbb{E}\| z^{t+\frac{1}{2}} - z^{t} \|^2 \notag.
    \end{align}
    Assume the following notation:
    \begin{align*}
        \Delta_\tau := { 2^{-\tau/\tau_{\mathrm{mix}}+1}} \sum\limits_{t=\tau}^{2\tau-1} p_t\mathbb{E} \| z^{t+\frac{1}{2}-\tau} - z^* \|^2 + 6\gamma(1+\gamma L)^2\beta_1\tau^2 \sum\limits_{t=0}^{\tau-1}p_t\mathbb{E}\|z^{t+\frac{1}{2}} - z^t\|^2.
    \end{align*}
    Then, rearranging \eqref{eq:th1}, one can obtain
    \begin{align*}
        \sum\limits_{t=\tau}^{T-1} p_t\mathbb{E}\| z^{t + 1}  - z^*\|^2 &\leq \sum\limits_{t=\tau}^{T-1}p_t\mathbb{E}\| z^t - z^*\|^2 - 2 \gamma \mu \sum\limits_{t=\tau}^{T-1} p_t\mathbb{E}\| z^{t+\frac{1}{2}} - z^* \|^2 \\&\quad+ \left(\frac{\gamma}{\beta_1} + {2\gamma^2 2^{-\tau/\tau_{\mathrm{mix}}}}\right) {\sigma_*}^2\sum\limits_{t=\tau}^{T-1} p_t \\&\quad+ {2^{-\tau/\tau_{\mathrm{mix}}+1}} \sum\limits_{t=2\tau}^{T-1} p_t\mathbb{E} \| z^{t+\frac{1}{2}-\tau} - z^* \|^2 \\
        &\quad+ (\gamma^2 L^2 - 1) \sum\limits_{t=\tau}^{T-1} p_t\mathbb{E}\| z^{t+\frac{1}{2}} - z^t \|^2\\
        &\quad+ 6\gamma(1+\gamma L)^2\beta_1 \tau^2 \sum\limits_{t = \tau}^{T-1} p_{t} \mathbb{E}\| z^{t+\frac{1}{2}} - z^{t} \|^2 + \Delta_{\tau}.
    \end{align*}
    It follows from Cauchy-Schwarz inequality \eqref{eq:cbs1} with $\beta = 1$, that:
    \begin{align*}
        -\gamma \mu \| z^{t+\frac{1}{2}} - z^* \|^2 \leq - \frac{\gamma \mu}{2} \| z^t - z^*\|^2 + \gamma \mu \| z^{t+\frac{1}{2}} - z^t \|^2.
    \end{align*}    
    Combining this with the fact that $${2^{-\tau/\tau_{\mathrm{mix}}+1}} \sum\limits_{t=2\tau}^{T-1} p_t\mathbb{E} \| z^{t+\frac{1}{2}-\tau} - z^* \|^2 \leq {12\cdot 2^{-\tau/\tau_{\mathrm{mix}}}} \sum\limits_{t=\tau}^{T-1} p_t\mathbb{E} \| z^{t+\frac{1}{2}} - z^* \|^2,$$ we get:
    \begin{align*}
        \sum\limits_{t=\tau}^{T-1} p_t\mathbb{E}\| z^{t + 1}  &- z^*\|^2 \leq \left(1 - \frac{\mu\gamma}{2}\right)\sum\limits_{t=\tau}^{T-1}p_t\mathbb{E}\| z^t - z^*\|^2 \\
        &+ {\left(12\cdot2^{-\tau/\tau_{\mathrm{mix}}} -  \mu\gamma\right)}\sum\limits_{t=\tau}^{T-1} p_t\mathbb{E}\| z^{t+\frac{1}{2}} - z^* \|^2 \\
        & + {\left(\frac{\gamma}{\beta_1} + 2\gamma^2 2^{-\tau/\tau_{\mathrm{mix}}}\right)} \sum\limits_{t=\tau}^{T-1} p_t{\sigma_*}^2 + \Delta_{\tau} \\
        & + \left(6\gamma(1+\gamma L)^2\beta_1 \tau^2 + \gamma^2 L^2 + \mu\gamma - 1\right) \sum\limits_{t=\tau}^{T-1} p_t\mathbb{E}\| z^{t+\frac{1}{2}} - z^t \|^2.
    \end{align*}
    Taking 
    {$$\gamma \leq \min\left\{\frac{1}{2L}, \frac{1}{4\mu}\right\},\quad\quad \tau \geq \tau_{\mathrm{mix}}\log \frac{12}{\mu\gamma},\quad\quad \beta_1 = \frac{1}{24\gamma\tau^2},$$}
    we obtain:
    \begin{align*}
        {12\cdot2^{-\tau/\tau_{\mathrm{mix}}} -  \mu\gamma \leq 0,}
    \end{align*}
    \begin{align*}
        6\gamma(1+\gamma L)^2\beta_1 \tau^2 + \gamma^2 L^2 + \mu\gamma - 1 \leq 0.
    \end{align*}
    Therefore, we can claim that:
    \begin{align*}
        \sum\limits_{t={\tau_{mix}}}^{T-1} p_t\mathbb{E}\| z^{t + 1} - z^*\|^2 \leq \sum\limits_{t={\tau_{mix}}}^{T-1} \bigg[ 1 - \frac{\gamma \mu}{2} \bigg]p_t\mathbb{E}\| z^t - z^*\|^2 + {25}\gamma^2{\tau^2_{mix}}\sum\limits_{t={\tau_{mix}}}^{T-1} p_t{\sigma_*}^2 + \Delta_{\tau_{mix}}.
    \end{align*}
    Finally, we substitute $p_t := \left(1 - \frac{\gamma \mu}{2}\right)^{-t}$:
    \begin{align*}
        \sum\limits_{t={\tau_{mix}}}^{T-1} \bigg[1 - \frac{\gamma \mu}{2} \bigg]^{-t} \mathbb{E}\| {z^{t+1}} - z^* \| &\leq \sum\limits_{t={\tau_{mix}}}^{T-1} \bigg[1 - \frac{\gamma \mu}{2} \bigg]^{-t+1} \mathbb{E}\| {z^{t}} - z^*\|^2 \\&\quad+ {25}\gamma^2{\tau^2_{mix}}\sigma_*^2\sum\limits_{t={\tau_{mix}}}^{T-1}\bigg[1 - \frac{\gamma \mu}{2} \bigg]^{-t} + \Delta_{\tau_{mix}};
    \end{align*}
    Upon removing the contracting terms, the following remains:
    \begin{align*}
        \mathbb{E}\| z^{T} - z^* \| &\leq \bigg[1 - \frac{\gamma \mu}{2} \bigg]^{T-{\tau_{mix}}} \mathbb{E}\| z^{\tau_{mix}} - z^*\|^2 + \bigg[1 - \frac{\gamma \mu}{2} \bigg]^{T} \Delta_{\tau_{mix}} \\&\quad+ {25}\gamma^2{\tau^2_{mix}}\sigma_*^2 \sum\limits_{t={\tau_{mix}}}^T \bigg[1 - \frac{\gamma \mu}{2} \bigg]^{T-t}.
    \end{align*}
    Next, we use that:
    \begin{equation*}
        \sum\limits_{t={\tau_{mix}}}^T \bigg[1 - \frac{\gamma \mu}{2} \bigg]^{T-t} = \sum\limits_{t=0}^{T-{\tau_{mix}}} \bigg[1 - \frac{\gamma \mu}{2} \bigg]^{t} \leq \sum\limits_{t=0}^{\infty} \bigg[1 - \frac{\gamma \mu}{2} \bigg]^{t} = \frac{2}{\gamma \mu},
    \end{equation*}
    and bound $\Delta_{\tau_{mix}}$:
    \begin{equation*}
        \Delta_{\tau_{mix}} \leq \sum\limits_{t=0}^{{\tau_{mix}}-1}\mathbb{E}\Big[ \|z^{t+\frac{1}{2}} - z^t\|^2 + \|z^{t+\frac{1}{2}} - z^*\|^2\Big],
        \end{equation*}
    leading us to:
    \begin{align*}
        \mathbb{E} \|z^{T+1} - z^*\|^2 \leq \Bigg(1 - \frac{\mu\gamma}{2}\Bigg)^T\Bigg[\Big(1-\frac{\mu\gamma}{2}\Big)^{-{\tau_{mix}}}\mathbb{E}\|z^{\tau_{mix}} - z^*\|^2 + \Delta_{\tau_{mix}}\Bigg] + \frac{{50}\gamma{\tau^2_{mix}}}{\mu}\sigma_*^2.
    \end{align*}
    This finishes the proof.
    \end{proof}
\begin{corollary}[Step tuning for Theorem \ref{theorem:1}]
    \label{cor:1}
    Under the conditions of Theorem \ref{theorem:1}, choosing $\gamma$ as 
    \begin{equation*}
        {\gamma \leq \min\Bigg\{ \frac{1}{2L}, \frac{1}{4\mu}, \frac{2\log(\max\{2, \frac{\mu^2r_{\tau_{mix}} T^2}{100\tau^2\sigma_*^2}\})}{\mu T} \Bigg\}},
    \end{equation*}
    {where $r_t := \|z^t - z^{*}\|^2$}, in order to achieve the $\epsilon$-approximate solution in terms of $\mathbb{E}\| z^T - z^*\|^2 \leq \epsilon^2$ it takes
    \begin{equation}
        \label{eq:estimation}
        {T = \Tilde{\mathcal{O}}\Bigg(\frac{L}{\mu}\log\frac{1}{\epsilon} + \frac{\tau^2_{mix}\sigma_*^2}{\mu^2\epsilon}\Bigg)} ~~~\text{oracle calls.}
    \end{equation}
\end{corollary}
\section{Discussion.}
In spite of the fact that the Markov noise setup in the context of a classical optimization problem has a quite wide representation in the literature, there is considerably less work for the VI case. To the best of our knowledge, there are only three existing works on the topic of VI with Markovian stochasticity. Two of them, \cite{Wang2022} and \cite{solodkin2024methodsoptimizationproblemsmarkovian}, consider only monotone setting, obtaining results of the form $T = \widetilde{\mathcal{O}}\left( \frac{L\sigma^2}{\epsilon^2} + \frac{{\tau^2_{mix}} \sigma^4}{\epsilon^2}\right)$ and $T = \widetilde{\mathcal{O}}\left(\frac{L D^2}{\epsilon} + \frac{{\tau_{mix}} D^2 \sigma^2}{\epsilon^2} \right)$ 
respectively. The third work \cite{Beznosikov2023MarkovianNoise} provides the following convergence guarantees: $T = \Tilde{\mathcal{O}}\left(\frac{{\tau_{mix}} L}{\mu}\log\frac{1}{\epsilon} + \frac{{\tau_{mix}}\sigma^2}{\mu^2\epsilon}\right)$. This result is nearly identical to that of Corollary \ref{cor:1}, with the exception of the mixing time entry. However, the authors of \cite{Beznosikov2023MarkovianNoise} utilize the batching technique with batches of size $\Tilde{\mathcal{O}}({\tau_{mix}})$, resulting in a significant increase in gradient evaluations. Moreover, this bound is contingent upon the assumption of uniformly bounded gradient differences, while our analysis is considerably less demanding with the necessity in only bound at the optimal point.  
\newpage
\section{Numerical Experiments}
In this section, we present numerical experiments that are designed to investigate the effect of mixing time on the convergence rate.
\subsection{Problem formulation}

 Let $\lambda$, $\nu$ be positive real numbers. Let $b,c \in \R^d$ be vectors with randomly generated from $[-1, 1]$ real components. Finally, define $P \in \mathbb{R}^{d \times d}$ as a matrix of randomly generated real values such that $\textbf{spec} (P) \subset [0.1, 10]$.

Introducing this notation, we consider the following formulation of the problem \eqref{eq:main}:
\begin{equation}
    \label{prob}
    \min\limits_{x \in \mathbb{R}^d} \max\limits_{y \in \mathbb{R}^d} \left[ f(x, y) := x^T P y + b^T x + c^T y + \frac{\lambda}{2} \| x \|_2^2 - \frac{\nu}{2} \| y \|_2^2 \right].
\end{equation}

For this problem, the operator has the following form:

$$F(z) := \begin{pmatrix}
\nabla_x f(x, y)\\
- \nabla_y f(x, y)
\end{pmatrix} = \begin{pmatrix}
Py + b + \lambda x\\
-P^Tx - c - \nu y
\end{pmatrix}$$

\subsection{Setup}

In the numerical experiments, we consider the problem described above on the different ergodic Markov chains.
% $\bullet$ \textbf{Brief Description}\\
% Suppose that a discrete Markov chain is generated with A as the initial state. Then the problem is solved by an algorithm that uses the state-dependent noise of the chain as stochasticity while calculating the gradient of the target function. With each iteration the states are changed according to the transition matrix. We perform these experiments for different parameters $p$, which will be defined later, and generate a unique Markov chain, allowing us to compare the convergence with chains with different mixing time parameters. We also provide comparison with changing a parameter of the noise distribution.
In order to compare the outcomes properly, we let all the Markov chains have the similar structure:

\begin{center}
\begin{tikzpicture}[scale=0.11]
\tikzstyle{every node}+=[inner sep=0pt]
\draw [black] (17.5,-17.5) circle (3);
\draw (17.5,-17.5) node {$A$};
\draw [black] (54.2,-17.5) circle (3);
\draw (54.2,-17.5) node {$B$};
\draw [black] (19.61,-15.37) arc (131.74942:48.25058:24.389);
\fill [black] (52.09,-15.37) -- (51.83,-14.46) -- (51.16,-15.21);
\draw (35.85,-8.68) node [above] {$1\mbox{ }-\mbox{ }p$};
\draw [black] (52.238,-19.767) arc (-44.60174:-135.39826:23.017);
\fill [black] (19.46,-19.77) -- (19.67,-20.69) -- (20.38,-19.99);
\draw (35.85,-27.12) node [below] {$1\mbox{ }-\mbox{ }p$};
\draw [black] (14.82,-18.823) arc (-36:-324:2.25);
\draw (10.25,-17.5) node [left] {$p$};
\fill [black] (14.82,-16.18) -- (14.47,-15.3) -- (13.88,-16.11);
\draw [black] (56.88,-18.823) arc (-144:144:2.25);
\draw (61.65,-17.5) node [right] {$p$};
\fill [black] (56.88,-16.18) -- (57.23,-15.3) -- (57.84,-16.11);
\end{tikzpicture}
\end{center}

Here $p$ is a unique parameter of the Markov chain, arrows denote the transitions, and $A$ and $B$ represent the states. Each state implies a unique current distribution that generates the noise. In our experiments, we assume that both of the states have a normal distribution, in particular, state $A$ generates a value from the distribution $\mathcal{N}(0.1, \sigma^2)$, state $B$ generates a value from $\mathcal{N}(-0.1, \sigma^2)$, where $\sigma$ is a varying parameter for deeper study. The generated noise $\xi$ is considered to be additive:
$$\left[F(z, \xi)\right]_i := \left[F(z)\right]_i + \xi_i.$$

\subsection{Results}
We first performed the experiments with $\sigma=0.001$ (Figure \ref{fig: std_0.001}), $\sigma=0.01$ (Figure \ref{fig: std_0.01}), $\sigma=0.1$ (Figure \ref{fig: std_0.1}) and $\sigma=0$ (Figure \ref{fig: std_0}) as the standard deviations of normal distributions described in the setup.

\begin{figure}[ht]
     \centering
     \subfloat[$\sigma = 0.001$]{% 
     \resizebox*{7cm}{!}{\includegraphics{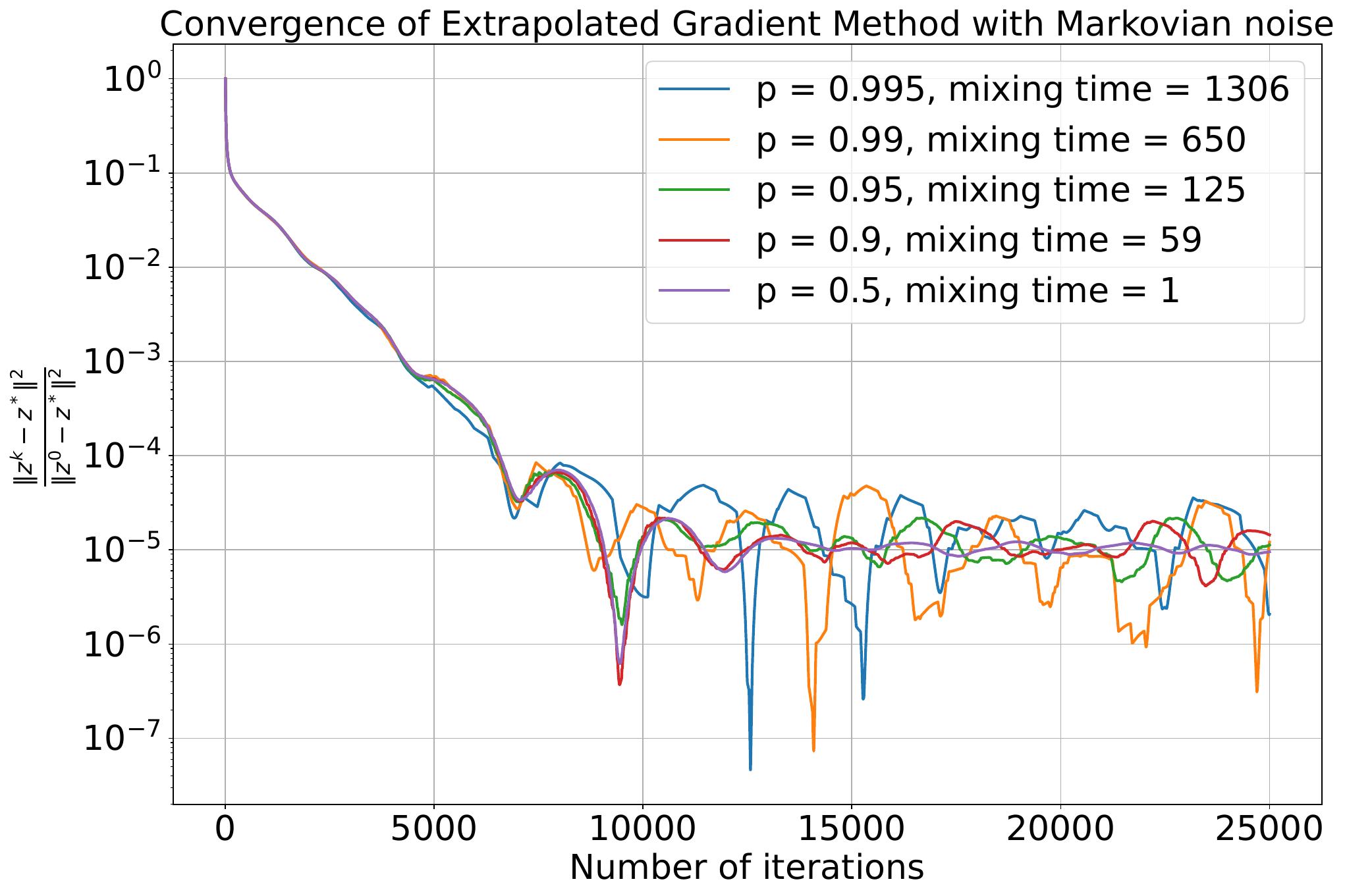}}\label{fig: std_0.001}}
     \hfill
     \subfloat[$\sigma = 0.01$]{% 
     \resizebox*{7cm}{!}{\includegraphics{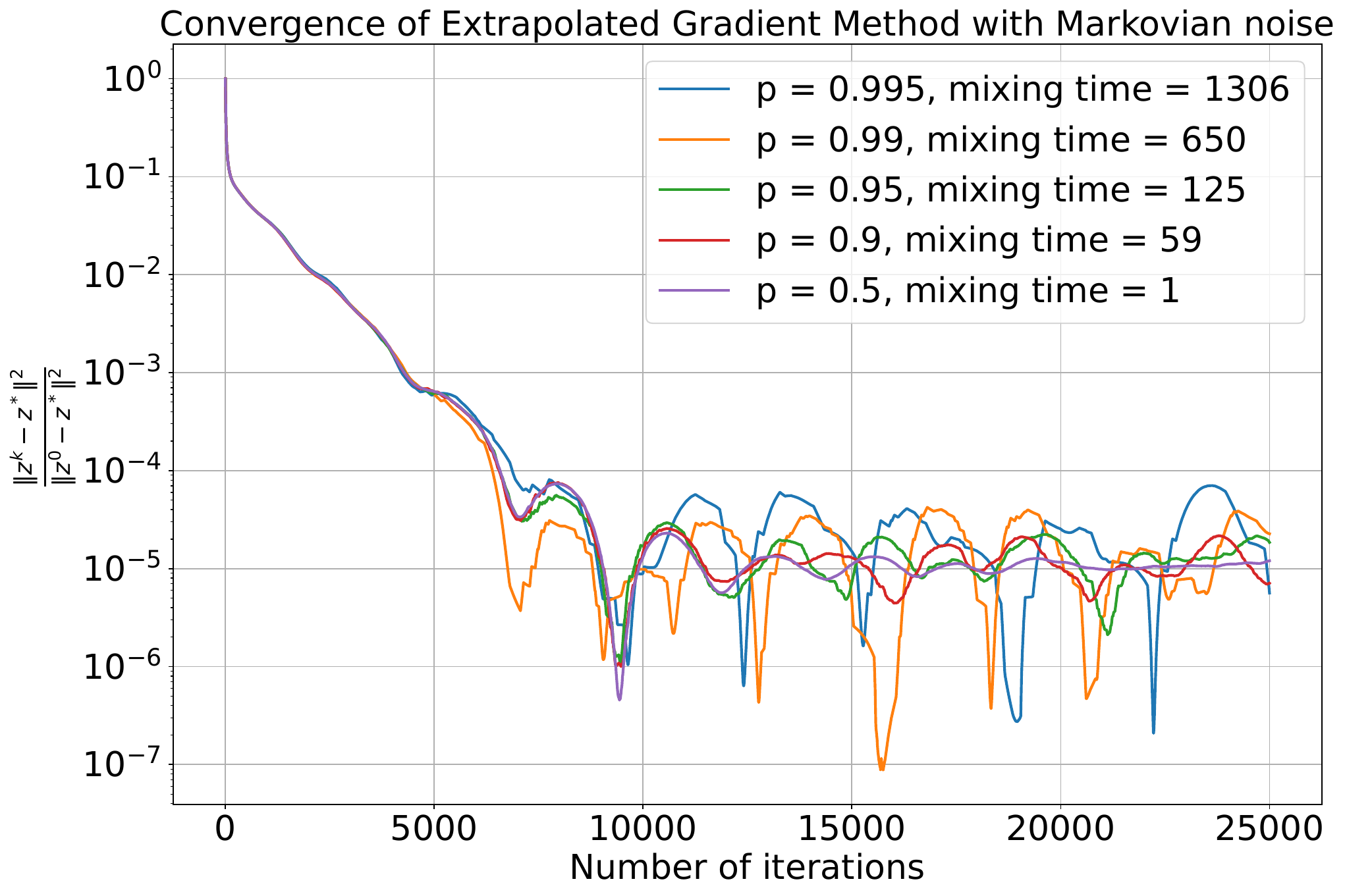}}\label{fig: std_0.01}}
     \vfill
     \subfloat[$\sigma = 0.1$]{% 
     \resizebox*{7cm}{!}{\includegraphics{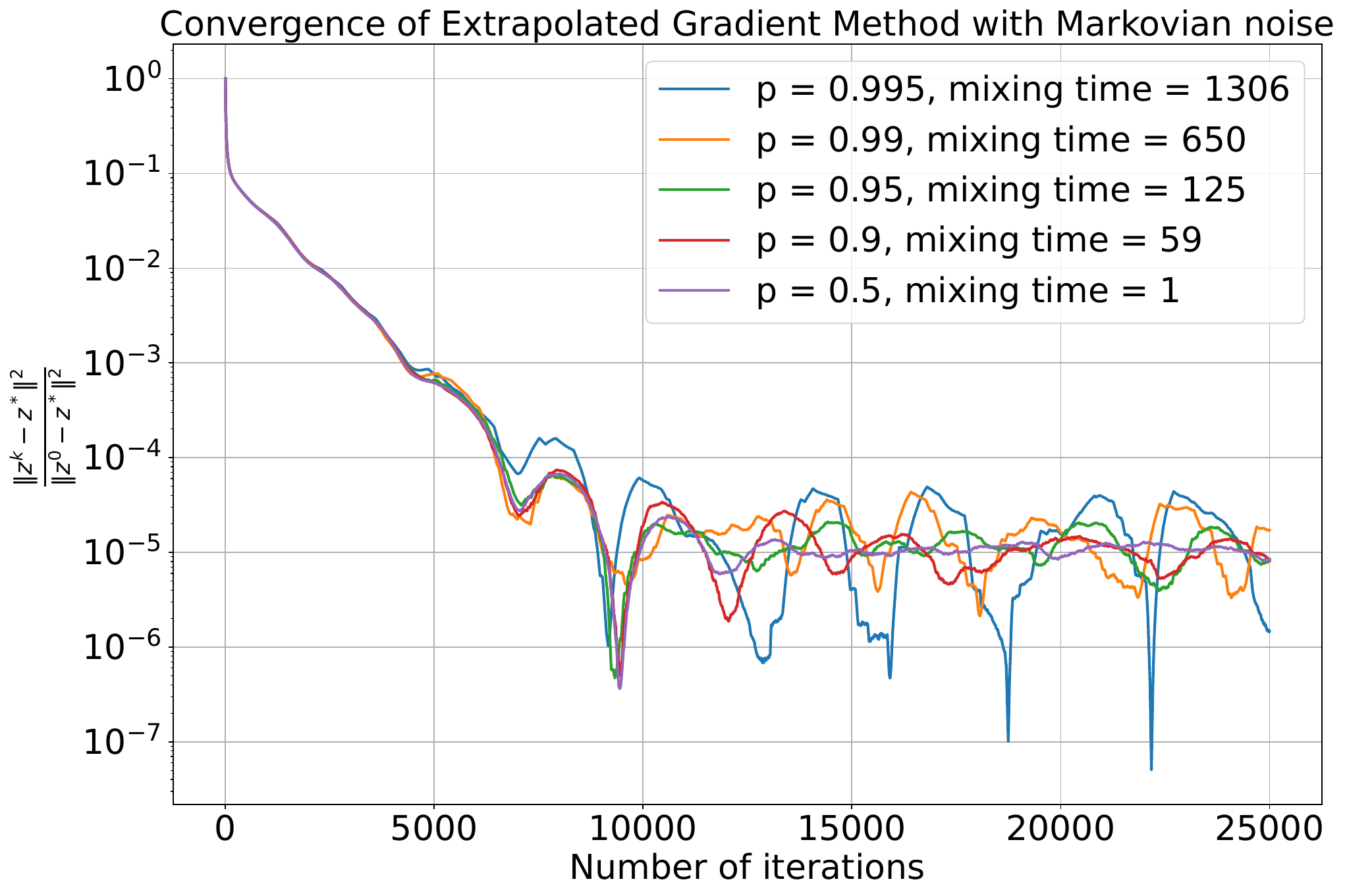}}\label{fig: std_0.1}}
     \hfill
     \subfloat[$\sigma = 0$]{% 
     \resizebox*{7cm}{!}{\includegraphics{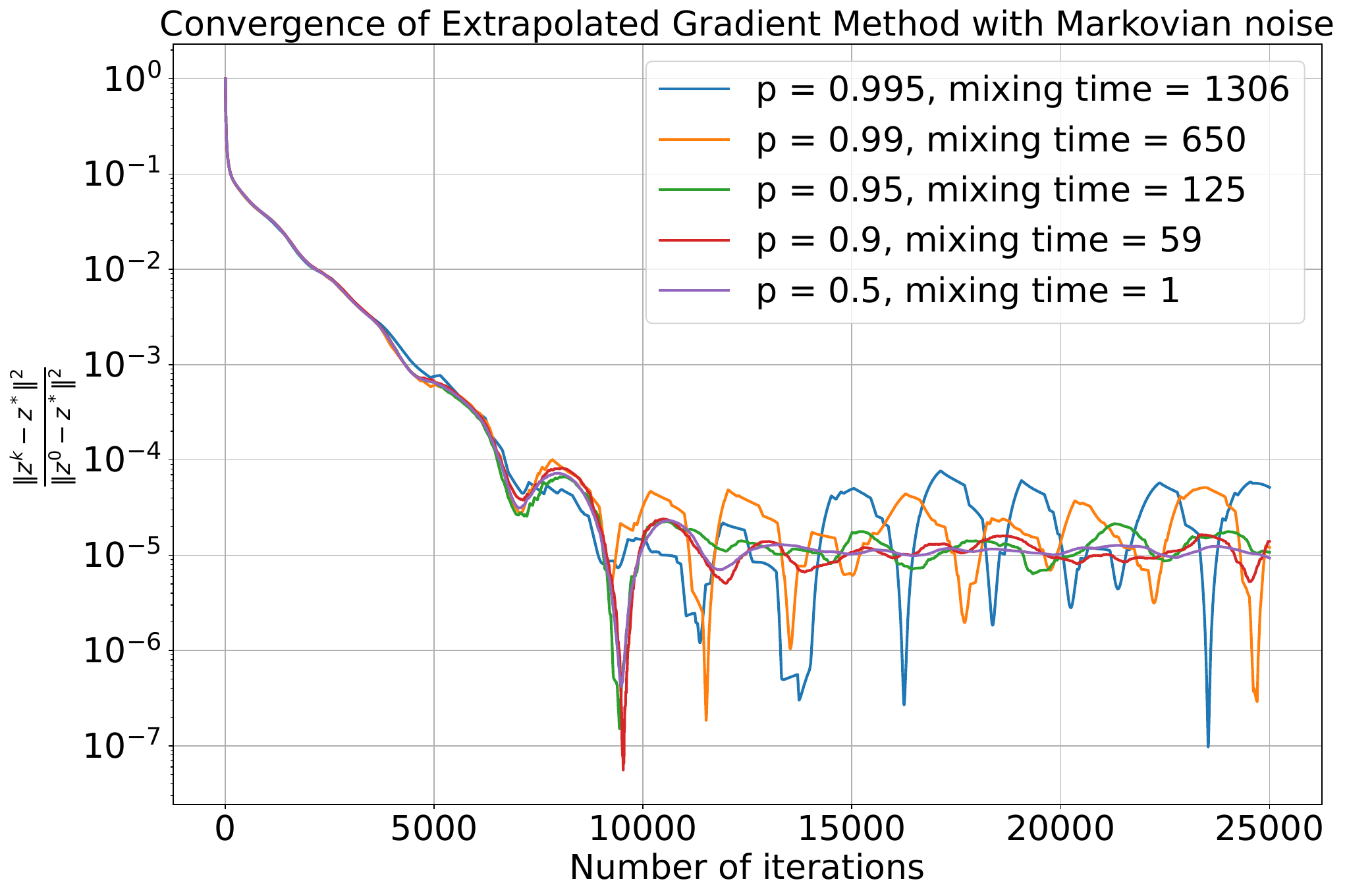}}\label{fig: std_0}}
     % \begin{subfigure}[b]{0.45\textwidth}
     %  \centering
     %  \includegraphics[width=\textwidth]{sigma=0.001, mean=0.1, lambda~0.01, b,c~1.pdf}
     %     \caption{$\sigma = 0.001$}
     %     \label{fig: std_0.001}
     % \end{subfigure}
     % \hfill
     % \begin{subfigure}[b]{0.45\textwidth}
     %     \centering
     %    \includegraphics[width=\textwidth]{sigma=0.01, mean=0.1, lambda~0.01, b,c~1.pdf}
     %     \caption{$\sigma = 0.01$}
     %     \label{fig: std_0.01}
     % \end{subfigure}
     % \vfill
     % \begin{subfigure}[b]{0.45\textwidth}
     %     \centering
     %     \includegraphics[width=\textwidth]{sigma=0.1, mean=0.1, lambda~0.01, b,c~1.pdf}
     %     \caption{$\sigma = 0.1$}
     %     \label{fig: std_0.1}
     % \end{subfigure}
     % \hfill
     % \begin{subfigure}[b]{0.45\textwidth}
     %     \centering
     %     \includegraphics[width=\textwidth]{sigma=0, mean=0.1, lambda~0.01, b,c~1.pdf}
     %     \caption{$\sigma = 0$}
     %     \label{fig: std_0}
     % \end{subfigure}
        \caption{Comparison of the performance of Algorithm \ref{alg:EGM} for varying mixing time parameter $\tau$ for the different values of the variance $\sigma$.}
        \label{fig:1}
\end{figure}
% \begin{figure}
% \centering
% \subfloat[An example of an individual figure sub-caption.]{%
% \resizebox*{5cm}{!}{\includegraphics{graph1.eps}}}\hspace{5pt}
% \subfloat[A slightly shorter sub-caption.]{%
% \resizebox*{5cm}{!}{\includegraphics{graph2.eps}}}
% \caption{Example of a two-part figure with individual sub-captions
%  showing that captions are flush left and justified if greater
%  than one line of text.} \label{sample-figure}
% \end{figure}
As illustrated in Figure \ref{fig:1}, the oscillation amplitude of the method is dependent on the mixing time. Furthermore, stochasticity does not impact this dependence, as evidenced by the consistent character of convergence observed in experiments with varying standard deviations. At this juncture, it seems appropriate to examine the nature of this dependence in greater detail. 

Now, we fix the values of standard deviation $\sigma = 0.1$ and expectation $\mu = 0.1$ for the noise variables. In order to most accurately assess the impact of the mixing time parameter on the convergence region, we ran Algorithm \ref{alg:EGM} $K = 14$ times for each value of $p$ with identical hyperparameters $\gamma = \frac{1}{2L}$. Subsequently, the sample variance was calculated for each solution, i.e. denoting the results of $j$-th run by the $\{z^j_t\}_{t = 0}^{T}$, we calculate $\frac{1}{K}\sum\limits_{j = 1}^{K}\bigg[\frac{1}{T}\sum\limits_{t = 1}^{T}\Big(z_t^j-\overline{z}^j\Big)^2\bigg]$, where $\overline{z}^j = \frac{1}{T}\sum\limits_{t = 1}^{T}z_t^j$. As a consequence, we were able to establish a reliable correlation between the convergence region and the mixing time, which serves to substantiate the conclusions reached in the theoretical analysis. However, the results of Figure \ref{fig: deviation} indicate that the nature of this dependence is tend to be linear, whereas the theoretical estimation \eqref{eq:estimation} implies the quadratic correlation. This raises a slight research gap, namely whether it is feasible to design a more intricate proof that would yield a more precise assessment of the dependence on mixing time. Alternatively, it may be possible to modify the experimental procedure in order to more clearly observe the effect of mixing time parameter. In any case, this seems to us as the perfect topic for future research.
\begin{figure}[ht]
    \centering
    \includegraphics[width=0.6\textwidth]{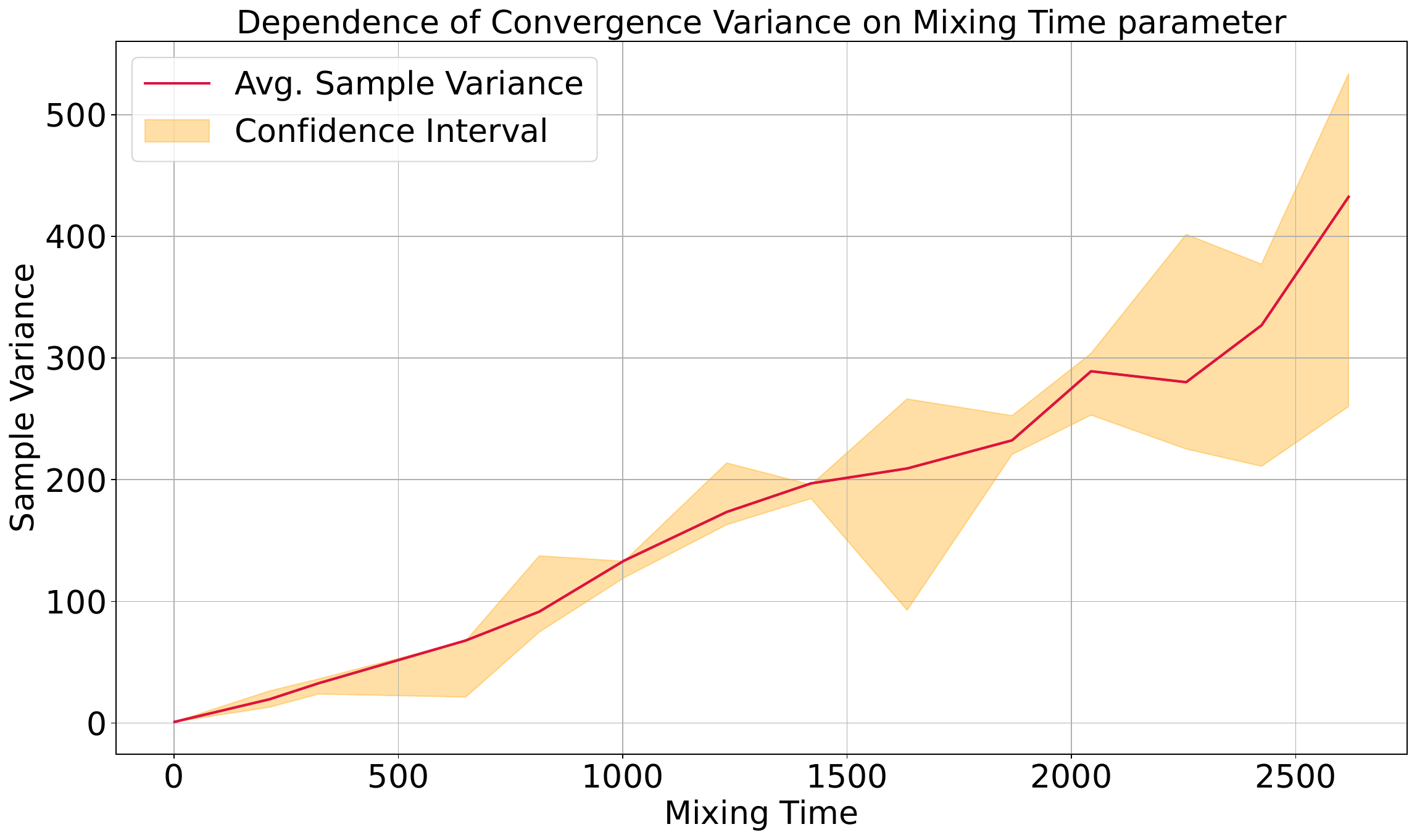}
    \caption{Comparison of convergence regions of Algoritm \ref{alg:EGM} for different values of mixing time. For the sake of clarity, the values are normalized on the variance for $\tau = 1$.}
    \label{fig: deviation}
\end{figure}

\section*{Acknowledgements}

The research was supported by Russian Science Foundation (project No. 25-71-00058)

\bibliographystyle{tfs}
\bibliography{references}

\appendix
\part*{Supplementary Material}

\section{Auxiliary Facts}
\begin{lemma}[Cauchy-Schwarz inequality]
    For any $a, b\in\mathbb{R}^d$ and $\beta > 0$ the following inequalities hold
    \begin{align}
        \label{eq:cbs1}
        2\langle a, b\rangle \leq \frac{\|a\|^2}{\beta} + \beta\|b\|^2,
    \end{align}
    \begin{align}
        \label{eq:cbs2}
        |\langle a, b\rangle| \leq \|a\|\|b\|.
    \end{align}
\end{lemma}
\end{document}